\documentclass[preprint,12pt]{elsarticle}

\usepackage{amsmath,amssymb,amsthm}

\theoremstyle{plain}
\newtheorem{theorem}{Theorem}[section]	
\newtheorem{lemma}{Lemma}[section]

\theoremstyle{definition}
\newtheorem{definition}{Definition}[section]
\newtheorem{remark}{Remark}[section]

\DeclareMathOperator{\Card}{Card}

\DeclareMathOperator{\Var}{Var}

\newcommand{\R}{\mathbb{R}}

\newcommand{\mS}{\mathcal{S}}
\newcommand{\mF}{\mathcal{F}}

\newcommand{\mG}{\mathcal{G}}
\newcommand{\mL}{\mathcal{L}}
\renewcommand{\qed}{\hfill{\tiny \ensuremath{\blacksquare} }}%
\renewcommand{\tilde}{\widetilde}

\newcommand{\Ep}{{\mathrm{E}}}

\renewcommand{\Pr}{{\mathrm{P}}}

\newcommand{\bP}{\mathbb{P}}
\newcommand{\bE}{\mathbb{E}}

\newcommand{\bG}{\mathbb{G}}
\newcommand{\eps}{\varepsilon}

\begin{document}
\begin{frontmatter}

\title{Empirical and multiplier bootstraps for suprema of empirical processes of increasing complexity, and related Gaussian couplings}

\author[label1]{Victor Chernozhukov}
\ead{vchern@mit.edu}
\author[label2]{Denis Chetverikov}
\ead{chetverikov@econ.ucla.edu}
\author[label3]{Kengo Kato}
\ead{kkato@e.u-tokyo.ac.jp}

\address[label1]{Department of Economics and Center for Statistics, MIT, 50 Memorial Drive, Cambridge, MA 02142, USA.}
\address[label2]{Department of Economics, UCLA, Bunche Hall, 8283, 315 Portola Plaza, Los Angeles, CA 90095, USA.}
\address[label3]{Graduate School of Economics, University of Tokyo, 7-3-1 Hongo Bunkyo-ku, Tokyo 113-0033, Japan.}

\begin{abstract}
We derive strong approximations to the supremum of the non-centered empirical process indexed by a possibly unbounded VC-type class of functions by the suprema of the Gaussian and bootstrap processes. 
The bounds of these approximations are non-asymptotic, which allows us to work with classes of functions whose complexity increases with the sample size.  The construction of couplings is not of the Hungarian type and is instead based on the Slepian-Stein methods and Gaussian comparison inequalities.  
The increasing complexity of classes of functions and non-centrality of the processes make the results useful for applications in modern nonparametric statistics (Gin\'{e} and Nickl \cite{GN15}), in particular allowing us to study the power properties of nonparametric tests using Gaussian and bootstrap approximations. 
\end{abstract}

\begin{keyword}
coupling \sep empirical process \sep multiplier bootstrap process \sep empirical bootstrap process \sep Gaussian approximation \sep supremum

\MSC[2010] 60F17 \sep 62E17 \sep 62G20



\end{keyword}

\end{frontmatter}

\section{Introduction}
Let $(S,\mathcal{S})$ be a measurable space, and let $X,X_1,\dots,X_n$ be a sequence of i.i.d. random variables taking values in $(S,\mathcal{S})$ with a common distribution $P$. We assume that $S$ is a separable metric space and $\mathcal{S}$ is its Borel $\sigma$-field. Let $\mF$ be a class of measurable functions $f:S\to \R$ with a measurable envelope $F:S\to\R$ satisfying $F(x)\geq \sup_{f\in\mF}|f(x)|$ for all $x\in S$.  Define the empirical process indexed by $\mF$:
\[
\mathbb{G}_n f=\frac{1}{\sqrt{n}}\sum_{i=1}^n (f(X_i)-Pf), \ f \in \mF, 
\]
where $Pf=\int f d P=\bE[f(X)]$. Let $e_1,\dots,e_n$ be independent standard Gaussian random variables independent of $X_{1}^{n} := \{ X_1,\dots,X_n\}$. Define the multiplier bootstrap process indexed by $\mF$:
\begin{equation}\label{eq: MB process}
\mathbb{G}_n^e f = \frac{1}{\sqrt{n}}\sum_{i=1}^n e_i (f(X_i)-P_n f), \ f \in \mF,
\end{equation}
where $P_n$ is the empirical measure with respect to $X_{1},\dots,X_{n}$; that is, $P_n f =n^{-1}\sum_{i=1}^n f(X_i)$ for $f \in \mF$. Let $N_1,\dots, N_n$ be a sequence of random variables multinomially distributed with parameters $n$ and (probabilities) $1/n,\dots,1/n$ that are independent of $X_{1}^{n}$. Define the empirical bootstrap process indexed by $\mF$:
\[
\mathbb{G}_n^* f = \frac{1}{\sqrt{n}}\sum_{i=1}^n (N_i-1)f(X_i), \ f \in \mF.
\]

Suppose that $\mF \subset \mathcal{L}^{2}(P)$ is a VC type class of functions (the definition of VC type classes is recalled in Section \ref{sec: couplings}) with $\sup_{f \in \mF} |Pf| < \infty$. 
Then $\mF$ is totally bounded with respect to the semimetric 
\[
e_{P}(f,g) = \sqrt{P(f-g)^{2}}, \ f,g \in \mF,
\]
and there exists a centered Gaussian process $G_P$ indexed by $\mF$ with uniformly $e_{P}$-continuous sample paths and covariance function 
\begin{equation}\label{eq: covariance function}
\bE[G_P(f) G_P(g)]=\text{Cov}(f(X),g(X)), \ f,g \in \mF.
\end{equation}

In this paper, for a given  functional $B: \mF \to \R$, we are interested in constructing couplings for
\begin{align}
Z=\sup_{f\in\mF} (B(f)+\bG_n f) \quad &\text{and} \quad \tilde{Z} \stackrel{d}{=} \sup_{f\in\mF} (B(f)+G_P f), \label{eq: emp process} \\
Z^e=\sup_{f\in\mF} (B(f)+\bG_n^e f) \quad &\text{and} \quad \tilde{Z}^e \stackrel{d \mid X_1^n}{=} \sup_{f\in\mF} (B(f)+G_P f), \label{eq: mult bootstrap process} \\
Z^*=\sup_{f\in\mF} (B(f)+\bG_n^* f) \quad  &\text{and} \quad \tilde{Z}^* \stackrel{d \mid X_1^n}{=} \sup_{f \in \mF} (B(f)+G_P f), \label{eq: emp bootstrap process}
\end{align}
such that the random variables appearing in each line are close to each other with high probability. The notation $\stackrel{d}{=}$ means equality in distribution, and $\stackrel{d \mid X_1^n}{=}$ means equality in conditional distribution given $X_{1}^{n}=\{ X_{1},\dots,X_{n} \}$. Here we suppose that the probability space is such that 
\[
(\Omega, \mathcal{A}, \bP) = (S^{n},\mS^{n},P^{n}) \times (T,\mathcal{T},Q) \times ([0,1],\mathcal{B}([0,1]),\lambda)
\]
where $X_1,\dots,X_n$ are the coordinate projections of $(S^{n},\mS^{n},P^{n})$, random variables $e_1,\dots,e_n$ (or $N_1,\dots,N_n$) depend on the ``second'' coordinate only, and $([0,1],\mathcal{B}([0,1]),\lambda)$ is the Lebesgue probability space on $[0,1]$, that is, $\mathcal{B}([0,1])$ is the Borel $\sigma$-field on $[0,1]$ and $\lambda$ is the Lebesgue measure on $[0,1]$. The last augmentation of the probability space enables us to generate a uniform random variable on $[0,1]$ independent of $X_{1},\dots,X_{n}$ and $e_{1},\dots,e_{n}$ (or $N_{1},\dots,N_{n}$).
We also implicitly assume here that the functional $B$ and the class $\mF$ are ``nice'' enough so that measurability problems do not arise; see Section \ref{sec: couplings} for explicit assumptions.

Our coupling constructions are based on the Slepian-Stein methods and Gaussian comparison inequalities and built on the ideas in \cite{C05a, C05b, CCK0, CCK1, CCK4, CCK2, CCK3}. We emphasize that the construction of couplings in this paper is non-asymptotic, and so the class of functions $\mF=\mF_n$ may depend on $n$, and its complexity may grow as the sample size increases. This feature of the couplings is especially important in modern nonparametric statistics \cite{GN15}; see \cite{CCK1} and \cite{CCK4} for examples of applications.

We also emphasize that our couplings are not of the Hungarian type, and so are different from those obtained in e.g. \cite{K94} and \cite{R94}, among many others. In particular, in contrast to e.g. \cite{R94}, our couplings do not depend on the maximal total variation in $\mF$. Instead, the couplings only depend on VC properties of the class of functions $\mF$ as well as on certain moments of the functions in $\mF$ and the envelope $F$. 
This feature of the construction leads to a different range of possible applications in comparison with Hungarian couplings; see the detailed discussion in \cite{CCK1}.

Gaussian and bootstrap approximations of the supremum of a {\em non-centered} empirical process have many potential applications. 
For example, these approximations can be used to derive non-asymptotic bounds on the errors in multivariate CLT. Specifically, let $S=\R^p$, and let $A$ be a closed convex set in $S$. For $\mathcal{V}^{p-1}=\{v\in \R^p: \|v\|=1\}$, let $V_A:\mathcal{V}^{p-1}\to \R$ be the support function of $A$ defined by $V_A(v)=\sup_{x\in A}v^{T}x$. Then $x\in A$ if and only if $\sup_{v\in \mathcal{V}^{p-1}}(v^{T}x - V_A(v))\leq 0$. Therefore, our results can be used to approximate
\begin{equation}
\label{eq: prob example}
\bP\Big(\frac{1}{\sqrt{n}}\sum_{i=1}^n X_i\in A\Big)=\bP\Big(\sup_{v\in \mathcal{V}^{p-1}}\Big(\frac{1}{\sqrt{n}}\sum_{i=1}^n v^{T}X_i - V_A(v)\Big)\leq 0\Big).
\end{equation}
Here, the dimension $p=p_n$ of the sample space $S=\R^p$ can depend on the sample size $n$ and increase as $n$ grows. Importantly, if the set $A$ is such that the set $\mathcal{V}^{p-1}$ on the right-hand side of (\ref{eq: prob example}) can be reduced to a sufficiently small subset of $\mathcal{V}^{p-1}$, the Gaussian approximation becomes possible even if $p$ is larger or much larger than $n$; see \cite{CCK0} and \cite{CCK3} for examples. More broadly, one can use our results for distributional approximation of general convex functionals on $\R^p$ where the probability measure on $\R^p$ is given by the distribution of a normalized sum of i.i.d. random vectors; see Section 11 of  \cite{DLS98} where it is demonstrated that such functionals can be represented as suprema of non-centered empirical processes. 

Another possible application is to study power properties of nonparametric tests where under the null, the statistic can be approximated by $\sup_{f\in\mF} \bG_n f$, and under the alternative, the statistic can be approximated by $\sup_{f\in  \mF}(B(f)+\bG_n f)$, the functional $B$ representing deviations from the null hypothesis.
Finally, non-centered empirical processes are useful fore multi-scale testing where one combines many statistics corresponding to different scales into one test using scale-dependent critical value for each statistic; see \cite{DS01} where such tests were used for qualitative hypotheses testing.

This paper builds upon but differs from our previous papers \cite{CCK0,CCK1,CCK4,CCK2,CCK3}. In particular, this paper establishes, in the infinite dimensional setting, formal results on the multiplier and empirical bootstraps when the envelope $F$ may be unbounded. In addition, this paper allows to approximate the supremum of a possibly non-centered empirical process. These settings are not covered in our previous papers \cite{CCK0,CCK1,CCK4,CCK2,CCK3} and are new. 

The organization of this paper is as follows. In the next section, we present our main coupling theorems. In Section \ref{sec: auxiliary couplings}, we derive two auxiliary theorems that deal with maxima of high-dimensional random vectors. All the proofs are deferred to Sections \ref{sec: proofs} and \ref{sec: proofs auxiliary couplings}. For convenience of the reader, we cite some additional results that are useful in our derivations in Section \ref{sec: technical tools}.

\subsection{Notation} 
We use standard notation from the empirical process literature. For any probability measure $Q$ on a measurable space $(S,\mS)$, we use the notation $Qf = \int f dQ$.
For $p\geq 1$, we use $\mL^{p}( Q )$ to denote the space of all measurable functions $f: S \to \R$ such that $\| f \|_{Q,p} = (Q | f |^{p})^{1/p} < \infty$.
We define the (semi)metric $e_{Q}$ on $\mL^{2}( Q )$ by
$e_{Q}(f,g) = \| f - g \|_{Q,2}, \ f,g \in \mL^{2}( Q )$.

For $\varepsilon > 0$, an $\varepsilon$-net of a (semi)metric space $(T,d)$ is a subset $T_{\varepsilon}$ of $T$ such that for every $t \in T$ there exists a point $t_{\varepsilon} \in T_{\varepsilon}$ with $d (t,t_{\varepsilon}) < \varepsilon$.  The $\varepsilon$-covering number $N(T,d,\varepsilon)$ of $T$ is the infimum of the cardinality of $\varepsilon$-nets of $T$, that is,  $N(T,d,\varepsilon) = \inf \{ \Card (T_{\varepsilon}) : \ \text{$T_{\varepsilon}$ is an $\varepsilon$-net of $T$} \}$. For a subset $A$ of a semimetric space $(T,d)$, we use $A^{\delta}$ to denote the $\delta$-enlargement of $A$, that is,  $A^{\delta} = \{ x \in T : d(x,A) \leq \delta \}$ where $d(x,A) = \inf_{y \in A}d(x,y)$. We also use the notation $\| \cdot \|_{T} = \sup_{t \in T} \| \cdot \|$. 

For a function $g:\R \to \R$, we write $\|g\|_\infty=\sup_{x\in\R}|g(x)|$, and assuming that $g$ is differentiable, we use $g'$ to denote the derivative of $g$. We denote by $C^k(\R)$ the space of $k$-times continuously differentiable functions on $\R$. 
For $a,b \in \R$, we use the notation $a \vee b = \max \{ a,b \}$.

\section{Main results}
\label{sec: couplings}

In this section, we construct couplings between random variables in (\ref{eq: emp process}), (\ref{eq: mult bootstrap process}), and (\ref{eq: emp bootstrap process}) when $\mF$ is a VC type class of functions. Recall the definition: 
\begin{definition}[VC type class]
Let $\mF$ be a class of measurable functions on a measurable space $(S,\mS)$, to which a measurable envelope $F$ is attached.
We say that $\mF$ is VC type with envelope $F$ if there are constants $A,v > 0$ such that
$\sup_{Q} N(\mF,e_{Q},\varepsilon \| F \|_{Q,2}) \leq (A/\varepsilon)^{v}$ for all $0 < \varepsilon \leq 1$, where the supremum is taken over all finitely discrete probability measures on $(S,\mS)$.
\end{definition}

Let $B: \mF \to \R$ be a given functional, and for $\eta>0$, let $N_{B}(\eta)$ be the minimal integer $N$ such that there exist $f_{1},\dots,f_{N} \in \mF$ with the property that for every $f\in \mF$, there exists $1 \leq j \leq N$ with $|B(f)-B(f_j)|<\eta$. We make the following assumptions.
\begin{enumerate}
\item[(A)]
There exists a  countable subset $\mG$ of $\mF$ such that for any $f \in \mF$, there exists a sequence $g_{m} \in \mG$ with $g_{m} \to f$ pointwise and $B(g_m)\to B(f)$.
\item[(B)] 
The class of functions $\mF$ is VC type with a  measurable envelope $F$ and constants $A\geq e$ and $v\geq 1$. 
\item[(C)]
There exist constants $b \geq \sigma>0$ and $q\in[4,\infty)$ such that $\sup_{f\in\mF}P|f|^k\leq \sigma^2 b^{k-2}$ for $k=2,3,4$, and $\|F\|_{P,q}\leq b$.
\end{enumerate}

Assumptions (B) and (C) guarantee that $\mF$ is totally bounded with respect to the semimetric $e_{P}$, and there exists a centered Gaussian process $G_{P}$ indexed by $\mF$ with uniformly $e_{P}$-continuous sample paths and covariance function given in (\ref{eq: covariance function}). 

Pick any $\eta > 0$ and put
\[
K_n = K_{n} (v,A,b,\sigma,B,\eta) = \log N_{B}(\eta) + v(\log n\vee\log (A b/\sigma)).
\]
The following theorem provides a coupling for  $Z$ and $\tilde{Z}$. 

\begin{theorem}[Coupling for the supremum of the empirical process]\label{thm: inf dim GAR}
Suppose that assumptions (A)--(C) are satisfied, and in addition suppose that $K_{n}^{3} \leq n$. Let $Z=\sup_{f\in\mF}(B(f)+\mathbb{G}_n f)$. Then for every $\gamma\in(0,1)$, there exists a random variable $\tilde{Z}\stackrel{d}{=} \sup_{f \in \mF} (B(f)+G_{P}f)$ such that
\[
\bP \{ |Z-\tilde{Z}|>C_1 (\eta + \delta^{(1)}_{n}) \} \leq C_{2} (\gamma + n^{-1}) \\
\]
where $C_1,C_2$ are positive constants that depend only on $q$, and 
\begin{equation}\label{eq: delta1}
\delta_{n}^{(1)} = \delta_{n}^{(1)} (v,A,b,\sigma,q,B,\eta,\gamma) = \frac{b K_n}{\gamma^{1/q} n^{1/2-1/q}}+\frac{(b \sigma^2 K_n^2)^{1/3}}{\gamma^{1/3}n^{1/6}}.
\end{equation}
\end{theorem}

The result in Theorem \ref{thm: inf dim GAR} is new because it allows for non-centered processes. In addition, even in the case of centered processes, i.e. when $B\equiv 0$, the bound here improves slightly on our previous result given in Corollary 2.2 of \cite{CCK1}.

\begin{remark}[Comparison with Beck's \cite{Beck85} lower bounds]
Suppose that $\mF$ is the class of indicators of closed balls in $\R^{d}$, and $X_{1},X_{2},\dots$ are i.i.d. uniform random variables on $[0,1]^{d}$. Then \cite{M89} proved, via KMT constructions, that there exist versions $B_{n}$ of $G_{P}$ such that 
\begin{equation}
\| \bG_{n} - B_{n} \|_{\mF} := \sup_{f \in \mF} | \bG_{n}f - B_{n}f| = O\{  n^{-1/(2d)}(\log n)^{3/2} \} \quad a.s., \label{eq: KMT rate}
\end{equation}
and up to a possible power of $\log n$, this rate is best possible when $d \geq 2$ (Beck \cite{Beck85}, Theorem 2). 

We shall apply Theorem \ref{thm: inf dim GAR} to this class of functions. In this example, $B \equiv 0$, and so $N_{B}(\eta) \equiv  1$. Since the class of closed balls in $\R^{d}$ is a VC class with index $d+2$ \citep[see][]{Du79}, assumption (B) is satisfied with $F \equiv 1, v = cd$ with some universal constant $c$, and $A$ being some universal constant. 
In addition,  assumption (C) is satisfied with $\sigma=b=1$ and arbitrary $q\in[4,\infty)$, so there is a universal constant $c'$ such that 
\[
\delta_{n}^{(1)} \leq c' \{ \gamma^{-1/q} d n^{-1/2+1/q} \log n +\gamma^{-1/3} d^{2/3}  n^{-1/6} (\log n)^{2/3} \}.
\]
If we take $\gamma=\gamma_{n} \to 0$ sufficiently slowly, say $\gamma_{n} = (\log n)^{-1/2}$, then for $Z_{n} = \sup_{f \in \mF} \bG_{n}f$, Theorem \ref{thm: inf dim GAR} implies that there exists a sequence $\tilde{Z}_{n}$ of random variables with $\tilde{Z}_{n} \stackrel{d}{=} \sup_{f \in \mF} G_{P}f$ such that 
\begin{equation}
\label{eq: coupiling balls}
|Z_{n} - \tilde{Z}_{n}| = o_{\mathbb{P}}\{ d n^{-1/2+1/q} (\log n)^{1+1/(2q)} + d^{2/3} n^{-1/6} (\log n)^{5/6}\}.
\end{equation}
This holds even when $d=d_{n} \to \infty$ as long as $d \log n= o(n^{1/3})$ (which guarantees $K_n^{3} \leq n$), and the right-hand side on (\ref{eq: coupiling balls}) is $o_{\bP}(1)$ if $d (\log n)^{5/4} = O(n^{1/4})$ by setting $q$ large enough. It is then clear that, although Theorem \ref{thm: inf dim GAR} is only applicable to the supremum, and there is a difference in the mode of convergence, the rate of approximation of our coupling in (\ref{eq: coupiling balls}) is better than that implied by (\ref{eq: KMT rate}) when $d$ is large. \qed
\end{remark} 

Next we provide a coupling for $Z^{e}$ and $\tilde{Z}^{e}$. 
\begin{theorem}[Coupling for the supremum of the multiplier bootstrap process]
\label{thm: inf dim MB}
Suppose that assumptions (A)--(C) are satisfied, and in addition suppose that $K_n \leq n$. Let $Z^e=\sup_{f\in\mF}(B(f)+\mathbb{G}_n^e f)$. Then for every $\gamma\in(0,1)$, there exists a random variable $\tilde{Z}^e \stackrel{d \mid X_{1}^{n}}{=} \sup_{f\in\mF}(B(f)+G_P f)$ such that
\[
\bP \{ |Z^e-\tilde{Z}^e|>C_3 (\eta + \delta^{(2)}_{n}) \} \leq C_4 (\gamma + n^{-1}), 
\]
where $C_3,C_4$ are positive constants that depend only on $q$, and 
\[
\delta_{n}^{(2)} = \delta_{n}^{(2)}(v,A,b,\sigma,q,B,\eta,\gamma) = \frac{b K_n}{\gamma^{1 + 1/q}n^{1/2-1/q}} + \frac{(b\sigma K_n^{3/2})^{1/2}}{\gamma^{1 + 1/q}n^{1/4}}. 
\]
\end{theorem}

Finally, we provide a coupling for $Z^*$ and $\tilde{Z}^*$.
\begin{theorem}[Coupling for the supremum of the empirical bootstrap process]
\label{thm: inf dim EB}
Suppose that assumptions (A)--(C) are satisfied, and in addition suppose that $K_n^3 \leq n$. Let $Z^*=\sup_{f\in\mF}(B(f)+\mathbb{G}_n^* f)$. Then for every $\gamma\in(0,1)$, there exists a random variable $\tilde{Z}^* \stackrel{d \mid X_{1}^{n}}{=} \sup_{f\in\mF}(B(f)+G_P f)$ such that
\[
\bP \{ |Z^*-\tilde{Z}^*|>C_5 (\eta + \delta^{(3)}_{n}) \} \leq C_6 (\gamma + n^{-1}), 
\]
where $C_5,C_6$ are positive constants that depend only on $q$, and 
\begin{align*}
\delta_{n}^{(3)} &= \delta_{n}^{(3)}(v,A,b,\sigma,q,B,\eta,\gamma) \\
& = \frac{b K_n}{\gamma^{1 + 1/q}n^{1/2-1/q}} + \frac{(b\sigma^2 K_n^{2})^{1/3}}{\gamma^{1/3} n^{1/6}} + \frac{(b \sigma K_n^{3/2})^{1/2}}{\gamma^{1 + 1/q}n^{1/4}}. 
\end{align*}
\end{theorem}

\begin{remark}
\label{rem: conditional coupling}
By Markov's inequality, the following inequality is directly deduced from Theorem \ref{thm: inf dim MB}: under the conditions of Theorem \ref{thm: inf dim MB}, for every $\alpha \in (0,1)$, with probability at least $1-\alpha$, we have
\[
\bP \{ |Z^{e} - \tilde{Z}^{e}| > C_3 (\eta + \delta^{(2)}_{n}) \mid X_{1}^{n} \} \leq \alpha^{-1} C_{4} (\gamma+n^{-1}). 
\]
Likewise,  the following inequality is directly deduced from Theorem \ref{thm: inf dim EB}: under the conditions of Theorem \ref{thm: inf dim EB}, for every $\alpha \in (0,1)$, with probability at least $1-\alpha$, we have
\[
\bP \{ |Z^{*} - \tilde{Z}^{*}| > C_5 (\eta + \delta^{(3)}_{n}) \mid X_{1}^{n} \} \leq \alpha^{-1} C_{6} (\gamma+n^{-1}). 
\]
\qed
\end{remark}

\begin{remark}
\label{rem: bounds in Kolmogorov distance}
In applications to statistics, it is often more useful to have  bounds on the Kolmogorov distance for the following pairs of distribution functions: $\bP(Z \leq \cdot)$ and $\bP(\tilde{Z} \leq \cdot)$; $\bP(Z^{e} \leq \cdot \mid X_{1}^{n})$ and $\bP(\tilde{Z} \leq \cdot)$; and $\bP(Z^{*} \leq \cdot \mid X_{1}^{n}) $ and $\bP(\tilde{Z} \leq \cdot)$. 
Once such bounds are obtained, we will have a bound on, say, the Kolmogorov distance between $\bP(Z \leq \cdot)$ and $\bP(Z^{e} \leq \cdot \mid X_{1}^{n})$. By the following simple lemma, we see that to obtain such bounds from the coupling inequalities stated in Theorems \ref{thm: inf dim GAR}--\ref{thm: inf dim EB}, we need an {\em anti-concentration} inequality for $\tilde{Z}$, that is, an inequality bounding $\sup_{t \in \R} \bP(|\tilde{Z} - t| \leq \eps)$ for $\eps > 0$. 
\begin{lemma}
Let $V,W$ be real-valued random variables such that $\bP(|V-W| > r_{1}) \leq r_{2}$ for some constants $r_{1},r_{2} > 0$. Then we have 
\[
\sup_{t \in \R} |\bP(V \leq t) - \bP(W \leq t)| \leq \sup_{t \in \R}\bP(| W-t| \leq r_{1}) + r_{2}.
\]
\end{lemma}
The proof of this lemma is immediate and hence omitted. In the case where $B(\cdot) \equiv 0$, a useful anti-concentration inequality for $\tilde{Z}$ is found in Lemma A.1 in \cite{CCK1}, which essentially follows from Theorem 3 in \cite{CCK2}. Lemma A.1 in \cite{CCK1} does not cover, however, non-centered Gaussian processes. Therefore, here we provide a new anti-concentration inequality that can be applied to non-centered Gaussian processes. The proof of the lemma can be found in Section \ref{sec: proofs}.
\begin{lemma}
\label{lem: AC for nonzero mean}
Let $T$ be a non-empty set, and let $\ell^{\infty}(T)$ be the set of all bounded functions on $T$ endowed with the sup-norm. Let $X(t), t \in T$ be a possibly non-centered tight Gaussian random element in $\ell^{\infty}(T)$ such that $\underline{\sigma}^{2} := \inf_{t \in T} \Var (X(t)) > 0$. 
Define $d(s,t) := \sqrt{\Ep[ (X(t)-X(s))^{2}]}, \ s,t \in T$, and for $\delta > 0$, define $\phi (\delta) :=  \bE [ \sup_{(s,t) \in T_{\delta}} | X(t) - X(s)|]$, where $T_{\delta} = \{ (s,t) : d(s,t) \leq \delta \}$. 
Then for every $\eps > 0$,
\begin{align*}
&\sup_{x \in \R} \bP (| \sup_{t \in T} X(t) -x| \leq \eps) \\
&\quad \leq \inf_{\delta,r > 0} \left \{  2(1/\underline{\sigma})(\eps + \phi(\delta) + r\delta) (\sqrt{2 \log N(T,d,\delta)} + 2) + e^{-r^{2}/2} \right \}.
\end{align*}
\end{lemma}
\qed
\end{remark}

\section{Auxiliary results for discretized processes}
\label{sec: auxiliary couplings}

This section states two auxiliary results for ``discretized'' processes that will be used to prove the theorems stated in Section \ref{sec: couplings}. 

\begin{theorem}
\label{thm: fin dim main}
Let $X_1,\dots,X_n$ be independent random vectors in $\R^p$ ($p \geq 2$) with finite absolute third moments, that is, $\bE[|X_{i j}|^3]<\infty$ for all $1\leq i\leq n$ and $1\leq j\leq p$. Define $\mu_i=\bE[X_i]$ and $\tilde{X}_i=X_i-\mu_i$, $1\leq i\leq n$, and consider the statistic $Z=\max_{1\leq j\leq p}n^{-1/2}\sum_{i=1}^n X_{i j}$. Let $Y_1,\dots,Y_n$ be independent random vectors in $\R^p$ with $Y_i\sim N(\mu_i,\bE[\tilde{X}_i \tilde{X}_i^T])$, and define $\tilde{Y}_i=Y_i-\mu_i$, $1\leq i\leq n$ and $\tilde{Z} = \max_{1 \leq j \leq p} n^{-1/2} \sum_{i=1}^{n} Y_{ij}$. Then for every $\delta>0$ and every Borel subset $A$ of $\R$, we have
\[
\bP (Z \in A) \leq \bP(\tilde{Z} \in A^{C_{7}\delta}) + \frac{C_8\log^2 p}{\delta^3 \sqrt{n}} \cdot \left \{ L_n+M_{n,X}(\delta)+M_{n,Y}(\delta) \right \},
\]
where $C_{7}, C_{8}$ are universal positive constants, and
\begin{align*}
&L_n=\max_{1\leq j\leq p}\frac{1}{n}\sum_{i=1}^n\bE\left[|\tilde{X}_{i j}|^3\right],\\
&M_{n,X}(\delta)=\frac{1}{n}\sum_{i=1}^n \bE\left[\max_{1\leq j\leq p}|\tilde{X}_{i j}|^3\cdot 1\left\{\max_{1\leq j\leq p}|\tilde{X}_{i j}|>\delta\sqrt{n}/\log p\right\}\right], \\
&M_{n,Y}(\delta)=\frac{1}{n}\sum_{i=1}^n \bE\left[\max_{1\leq j\leq p}|\tilde{Y}_{i j}|^3\cdot 1\left\{\max_{1\leq j\leq p}|\tilde{Y}_{i j}|>\delta\sqrt{n}/\log p\right\}\right].
\end{align*}
\end{theorem}

\begin{theorem}
\label{thm: mb finite dimensional}
Let $X=(X_1,\dots,X_p)^{T}$ and $Y=(Y_1,\dots,Y_p)^{T}$ be random vectors in $\R^p$ ($p \geq 2$) with  $X\sim N(\mu,\Sigma^{X})$ and $Y\sim N(\mu,\Sigma^{Y})$. 
Let $\Delta = \max_{1\leq j,k\leq p}|\Sigma_{jk}^{X} - \Sigma^{Y}_{jk}|$, where $\Sigma_{jk}^{X}$ and $\Sigma_{jk}^{X}$ denote the $(j,k)$-th elements of $\Sigma^{X}$ and $\Sigma^{Y}$, respectively. 
Define $Z=\max_{1\leq j\leq p}X_j$ and $\tilde{Z} =  \max_{1\leq j\leq p} Y_j$. Then for every $\delta>0$ and every Borel subset $A$ of $\R$, 
\[
\bP (Z \in A) \leq \bP(\tilde{Z} \in A^{\delta}) + C_9\delta^{-1}\sqrt{\Delta \log p},
\]
where $C_{9} > 0$ is a universal constant. 
\end{theorem}

\section{Proofs for Section \ref{sec: couplings}}
\label{sec: proofs}

Recall the definition of $K_{n}$:
\[
K_n = K_{n} (v,A,b,\sigma,B,\eta) = \log N_{B}(\eta) + v(\log n\vee\log (A b/\sigma)).
\]

\subsection{Proof of Theorem \ref{thm: inf dim GAR}}

The proof relies on the following form of Strassen's theorem.

\begin{lemma}[Strassen's theorem]\label{thm: strassen}
Let $\mu$ and $\nu$ be Borel probability measures on $\R$. Let $\varepsilon>0$ and $\delta>0$. Suppose that $\mu(A)\leq \nu(A^{\delta})+\varepsilon$ for every Borel subset $A$ of $\R$. Let $V$ be a random variable with distribution $\mu$. Then there is a random variable $W$ with distribution $\nu$ such that $\Pr(|V-W|>\delta)\leq \varepsilon$.
\end{lemma}
\begin{proof}[Proof of Lemma \ref{thm: strassen}]
See Lemma 4.1 in \cite{CCK1}.
\end{proof}

\begin{proof}[Proof of Theorem \ref{thm: inf dim GAR}]
By Strassen's theorem, it is sufficient to prove that for every Borel subset $A$ of $\R$, 
\begin{equation}
\bP(Z \in A) \leq \bP\{  \tilde{Z} \in A^{C_{1} (\eta+\delta^{(1)}_{n})} \} + C_{2}(\gamma+n^{-1}),  \label{eq: desired inequality1}
\end{equation}
where $\tilde{Z} = \sup_{f \in \mF} (B(f)+G_{P}f)$. The rest of the proof is divided into several steps. In the following, $C$ denotes a positive constant that depends only on $q$; the value of $C$ may change from place to place. 

\medskip

{\bf Step 1}. 
The first step is to ``discretize'' the empirical and Gaussian processes. 
To this end, take
\[
 \varepsilon=\sigma/(b n^{1/2}), \ N=2\cdot N(\mF,e_P, \eps b)\cdot N_B(\eta).
\]
Since $N(\mF,e_P, \eps b) \leq (4A/\eps)^{v}$ by approximation of $P$ by a finitely discrete probability measure and assumption (B), we have $\log N \leq CK_{n}$.
By definition, there exist $f_{1},\dots,f_{N} \in \mF$ such that for every $f\in\mF$, there exists $1\leq j\leq N$ with $e_P(f,f_j) < \eps b$ and $|B(f)-B(f_j)|<\eta$. 
Note that under the present assumption, the Gaussian process $G_{P}$ can be extended to the linear hull of $\mF$ in such a way that $G_{P}$ has linear sample paths \citep[see][Theorem 3.1]{D99}. 
Hence letting $\mF_\eps := \{f-g:f,g\in\mF, e_P(f,g)<\eps b\}$, 
we conclude that
\begin{align*}
&0\leq \sup_{f\in\mF}(B(f)+\mathbb{G}_n f)-\max_{1\leq j\leq N}(B(f_j)+\mathbb{G}_n f_j)\leq \eta+\|\mathbb{G}_n\|_{\mF_\eps}, \\
&0\leq \sup_{f\in\mF}(B(f)+G_P f)-\max_{1\leq j\leq N}(B(f_j)+G_P f_j)\leq \eta+\|G_P\|_{\mF_\eps}. 
\end{align*}

\medskip

{\bf Step 2}. Here we wish to show that 
\begin{equation}
\bP \{ \|G_P\|_{\mF_\eps} > C \sqrt{\sigma^2 K_n/n} \} \leq 2n^{-1}.
\end{equation}
This follows from the Borell-Sudakov-Tsirel'son inequality \citep[see][Proposition A.2.1]{VW96} complemented with Dudley's maximal inequality for Gaussian processes \citep[see][Corollary 2.2.8]{VW96}. 

First, by the Borell-Sudakov-Tsirel'son inequality, we have
\[
\bP \{ \|G_P\|_{\mF_\eps} > \bE[\|G_P\|_{\mF_\eps}] + \eps b \sqrt{2\log n} \} \leq 2n^{-1}.
\]
Second, by Dudley's maximal inequality together with the fact that $N(\mF_{\eps},e_{P},\tau) \leq N^{2}(\mF,e_{P},\tau/2) \leq (8Ab/\tau)^{2v}$, we have 
\[
\bE[\|G_P\|_{\mF_\eps}] \leq C \eps b \sqrt{v \log (8Ab/\eps)} \leq C \sqrt{\sigma^{2} K_{n}/n}.
\]
Combining these inequalities, together with the fact that $\log n \leq K_n$, leads to the desired inequality. 

\medskip

{\bf Step 3}. We wish to show that
\begin{equation}
\label{eq: emp process bound Step 3 - 1}
\bP \left \{  \|\bG_n\|_{\mF_\varepsilon} >  C b K_n/(\gamma^{1/q}n^{1/2-1/q}) \right \} \leq \gamma.
\end{equation}
Applying Lemma \ref{concentration} with $\alpha=\gamma^{-1/q}$ and $t=\gamma^{-2/q}$ to $\mF_{\eps}$, we have with probability at least $1-\gamma$, 
\[
\|\bG_n\|_{\mF_{\varepsilon}} \leq C \{ \gamma^{-1/q}\bE[\|\bG_n\|_{\mF_\varepsilon}] + (\sigma_\varepsilon+n^{-1/2}\|M_\varepsilon\|_q)\gamma^{-1/q}+n^{-1/2}\|M_\varepsilon\|_2\gamma^{-1/q} \},
\]
where $\sigma_{\varepsilon} := \sup_{f\in\mF_{\varepsilon}}(Pf^2)^{1/2}\leq \varepsilon b=\sigma/n^{1/2}$ and $M_{\eps} := 2\max_{1\leq i\leq n}F(X_i)$. 
Here $\|M_\varepsilon\|_2\leq \|M_\varepsilon\|_q \leq  2n^{1/q} b$. In addition, by Lemma \ref{cor: maximal}, we have
\[
\bE[\|\bG_n\|_{\mF_{\varepsilon}}] \leq C \{ \sigma(K_n/n)^{1/2}+bK_n/n^{1/2-1/q} \}\leq Cb K_n/n^{1/2-1/q}.
\]
Combining these inequalities leads to (\ref{eq: emp process bound Step 3 - 1}).

\medskip

{\bf Step 4}. Let $Z^\varepsilon=\max_{1\leq j\leq N}(B(f_j)+\bG_n f_j)$ and $\tilde{Z}^{\eps} = \max_{1 \leq j \leq N} (B(f_{j}) + G_{P}(f_{j}))$. Here we apply Theorem \ref{thm: fin dim main} to show that whenever
\begin{equation}\label{eq: delta bound step 3 - 1}
\delta\geq 2c\sigma n^{-1/2}(\log N)^{3/2}\cdot(\log n)
\end{equation}
for some universal constant $c>0$, we have for every Borel subset $A$ of $\R$, 
\begin{equation}\label{eq: Step 4 - 1}
\bP (Z^\varepsilon \in A) \leq \bP(\tilde{Z}^{\eps} \in A^{C_7\delta}) + C \left ( \frac{b\sigma^2 K_n^2}{\delta^3\sqrt{n}} + \frac{b^q K_n^{q}}{\delta^q n^{q/2-1}} + \frac{1}{n} \right ).
\end{equation}

Let $\tilde{X}_{i}=(f_j(X_i)-Pf_j)_{1 \leq j \leq N}, 1 \leq i \leq n$, and let $\tilde{Y} = (G_{P}f_{j})_{1 \leq j \leq N}$.
Then as $\tilde{X}_{1},\dots,\tilde{X}_{n}$ are i.i.d., 
\begin{align*}
L_{n} &= \max_{1\leq j\leq N}\bE[|\tilde{X}_{1 j}|^3] = \sup_{f \in \mF}\bE[|f(X)-Pf|^3] \leq 8 \sup_{f \in \mF} P|f|^{3} \leq 8\sigma^{2} b, \\
M_{n,X}(\delta)
&=\bE\left[\max_{1\leq j\leq N}|\tilde{X}_{1 j}|^3\cdot 1\left\{\max_{1\leq j\leq N}|\tilde{X}_{1 j}|>\delta\sqrt{n}/\log N\right\}\right]\\
&\leq \frac{\log^{q-3}N }{(\delta\sqrt{n})^{q-3}}\bE\left[\max_{1\leq j\leq N}|\tilde{X}_{1 j}|^q\right] \leq \frac{2^{q} b^q\log^{q-3}N }{(\delta\sqrt{n})^{q-3}}.
\end{align*}

To bound $M_{n,Y}(\delta)$, let $\| \cdot \|_{\psi_{1}}$ denote the Orlicz norm associated with the Young modulus $\psi_{1}(x) = e^{x}-1$, that is, $\| \xi \|_{\psi_{1}} = \inf \{ u > 0 : \bE[\psi_1 (|\xi|/u) ] \leq 1 \}$. Then it is routine to verify that there exists a universal constant $c > 0$ such that $\| \max_{1\leq j\leq N}|\tilde{Y}_{j}| \|_{\psi_{1}}  \leq c\sigma \sqrt{\log N}$.
Hence, by Markov's inequality, for every $x > 0$, 
\[
\bP \left(\max_{1\leq j\leq N}|\tilde{Y}_{j}|>x \right) \leq 2\exp\left(-\frac{x}{c \sigma\sqrt{\log N}}\right).
\]
Therefore, by Lemma \ref{lem: truncated moment}, whenever $\delta\geq 2c\sigma n^{-1/2}(\log^{3/2}N)\cdot (\log n)$,
\begin{align*}
M_{n,Y}(\delta)
&=\bE\left[\max_{1\leq j\leq N}|\tilde{Y}_{j}|^3\cdot 1\left\{\max_{1\leq j\leq N}|\tilde{Y}_{j}|>\delta\sqrt{n}/\log N \right\}\right]\\
&\leq 12(\delta\sqrt{n}/\log N + c\sigma\sqrt{\log N})^3\exp\left(-\frac{\delta\sqrt{n}}{c\sigma \log^{3/2}N}\right)\\
&\leq C n^{-2} (\delta\sqrt{n}/\log N)^3. 
\end{align*}
Application of Theorem \ref{thm: fin dim main} with these bounds, together with the bound $\log N \leq CK_{n}$,  leads to (\ref{eq: Step 4 - 1}).

\medskip

{\bf Step 5}. In the previous step, take
\[
\delta = C' \left \{ \frac{(b\sigma^2 K_n^2)^{1/3}}{\gamma^{1/3}n^{1/6}}+\frac{b K_n}{\gamma^{1/q}n^{1/2-1/q}}\right \},
\]
where $C' > 0$ is a large enough but universal constant. 
It is easy to check that for this choice of $\delta$, (\ref{eq: delta bound step 3 - 1}) holds under the condition $K_n^3\leq n$.
Indeed, since $q\geq 4$, $b \geq \sigma$, $\log n \leq  K_n$ and $\log N \leq C K_n$, we have $2c \sigma n^{-1/2}(\log^{3/2} N)\cdot (\log n) \leq C'\sigma K_n^{3/2}/n^{4/9} \leq C' b^{1/3}\sigma^{2/3}K_n^{2/3}/(\gamma^{1/3}n^{1/6})\leq \delta $. Therefore, by Step 3, we have for every Borel subset $A$ of $\R$, 
\[
\bP ( Z^{\eps} \in A ) \leq \bP (\tilde{Z}^{\eps} \in A^{C_{7} \delta}) + C(\gamma + n^{-1}).
\]
The desired inequality (\ref{eq: desired inequality1}) thus follows from combining Steps 1-5. 
\end{proof}

\subsection{Proof of Theorem \ref{thm: inf dim MB}}

The proof of Theorem \ref{thm: inf dim MB} relies on a conditional version of Strassen's theorem due to \cite{MP91}.

\begin{lemma}
\label{lem: conditional strassen}
Let $V$ be a real-valued random variable defined on a probability space $(\Omega,\mathcal{A},\bP)$, and let $\mathcal{C}$ be a countably generated sub $\sigma$-field of $\mathcal{A}$. 
Assume that there exists a uniform random variable on $[0,1]$ independent of $\mathcal{C} \vee \sigma (V)$. Let $G(\cdot \mid \mathcal{C})$ be a regular conditional distribution  on the Borel $\sigma$-field of $\R$ given $\mathcal{C}$,  and suppose that for some $\delta > 0$ and $\eps > 0$,
\[
\bE\left[ \sup_{A} \{ \bP(V \in A \mid \mathcal{C}) - G(A^{\delta} \mid \mathcal{C}) \} \right] \leq \eps,
\]
where $\sup_{A}$ is taken over all Borel subsets $A$ of $\R$. Then there exists a random variable $W$ such that the  conditional distribution of $W$ given $\mathcal{C}$ coincides with $G(\cdot \mid \mathcal{C})$, and moreover $\bP (|V-W| > \delta) \leq \eps$. 
\end{lemma}

\begin{proof}
See Theorem 4 in \cite{MP91}.
\end{proof}

\begin{proof}[Proof of Theorem \ref{thm: inf dim MB}]
Here $C$ denotes a positive constant that depends only on $q$; the value of $C$ may change from place to place. In addition, to ease the notation, we write $a \lesssim b$ if $a \leq Cb$. 
By Lemma \ref{lem: conditional strassen}, since $\sigma(X_{1}^{n})$ is countably generated by the construction of the probability space (in particular, recall that we have assumed that $S$ is a separable metric space), it is sufficient to find an event  $E \in \sigma (X_{1}^{n})$ such that $\bP(E) \geq 1-\gamma-n^{-1}$, and on this event, the inequality
\begin{equation}
\bP( Z^{e} \in A \mid X_{1}^{n}) \leq \bP \{ \tilde{Z} \in A^{C (\eta + \delta_{n}^{(2)})} \} + C (\gamma + n^{-1}) \label{eq: desired inequality2}
\end{equation}
holds for every Borel subset $A$ of $\R$, where $\tilde{Z} = \sup_{f \in \mF} (B(f)+G_{P}f)$.

We first specify such an event, and then show that on this event, (\ref{eq: desired inequality2}) holds for every Borel subset $A$ of $\R$. 
Applying Lemma \ref{concentration} with $\alpha=\gamma^{-1/q}$ and $t=(\gamma/2)^{-2/q}$ to $\mF$, we have with probability at least $1-\gamma/2$,
\[
\|\bG_n\|_{\mF} \lesssim \gamma^{-1/q}\bE[\|\bG\|_{\mF}] + (\sigma + n^{-1/2}\|M\|_q)\gamma^{-1/q}+n^{-1/2}\|M\|_2\gamma^{-1/q},
\]
where $M:=\max_{1\leq i\leq n}F(X_i)$ satisfies $\|M\|_2\leq \|M\|_q=(\bE[|M|^q])^{1/q}\leq n^{1/q} b$. In addition, by Lemma \ref{cor: maximal},
\[
\bE[\|\bG_n\|_\mF] \lesssim \sigma K_n^{1/2}+\|M\|_2K_n n^{-1/2}\leq \sigma K_n^{1/2} + b K_n n^{-1/2+1/q}.
\]
Hence with probability at least $1-\gamma/2$,
\begin{equation}\label{eq: event1}
\|\bG_n\|_{\mF} \lesssim \sigma K_n^{1/2}/\gamma^{1/q}+b K_n/(\gamma^{1/q}n^{1/2-1/q}).
\end{equation}
Moreover, applying Lemma \ref{concentration} again with $\alpha=\gamma^{-2/q}$ and $t=(\gamma/2)^{-4/q}$ to the class $\mF\cdot\mF := \{f\cdot g:f,g\in\mF\}$, we have with probability at least $1-\gamma/2$,
\[
\|\bG\|_{\mF\cdot\mF}\lesssim \gamma^{-2/q}\bE[\|\bG_n\|_{\mF\cdot\mF}]+(\bar{\sigma}+n^{-1/2} \|M^2\|_{q/2})\gamma^{-2/q}+n^{-1/2}\|M^2\|_2\gamma^{-2/q},
\]
where $\bar{\sigma}^2 := \sup_{f\in\mF\cdot \mF}P f^2 \leq \sup_{f\in\mF}P f^4\leq b^2\sigma^2$. 
In addition, $\| M^2\|_2\leq \|M^2\|_{q/2}=(\bE[|M|^q])^{2/q}\leq n^{2/q} b^2$, and as shown in the proof of Corollary 2.2 in \cite{CCK1},
$\bE[\|\bG_n\|_{\mF\cdot\mF}]\lesssim b\sigma K_n^{1/2}+b^2K_n n^{-1/2+2/q}$.
Hence with probability at least $1-\gamma/2$,
\begin{equation}\label{eq: event2}
\|\bG_n\|_{\mF\cdot\mF}\lesssim b\sigma K_n^{1/2}/\gamma^{2/q}+b^2K_n/(\gamma^{2/q}n^{1/2-2/q}).
\end{equation}
Finally, by Markov's inequality, with probability at least $1-n^{-1}$,
\begin{equation}\label{eq: event3}
\|F\|_{P_n,2}\leq n^{1/2} \|F\|_{P,2}. 
\end{equation}

Define $E$ as the intersection of the events in (\ref{eq: event1}), (\ref{eq: event2}), and (\ref{eq: event3}). Then $E \in \sigma(X_{1}^{n})$ and  $\bP(E) \geq 1-\gamma-n^{-1}$. 
The rest of the proof, which is divided into several steps, is devoted to proving (\ref{eq: desired inequality2}) for each fixed $X_{1},\dots,X_{n}$ satisfying (\ref{eq: event1})--(\ref{eq: event3}). 

In the following, we use the notation introduced in Step 1 of the proof of Theorem \ref{thm: inf dim GAR}. 
Then
\begin{align}
&0\leq \sup_{f\in\mF}(B(f)+\mathbb{G}_n^e f)-\max_{1\leq j\leq N}(B(f_j)+\mathbb{G}_n^e f_j)\leq \eta+\|\mathbb{G}_n^e\|_{\mF_\eps},\label{eq: MB appr1}\\
&0\leq \sup_{f\in\mF}(B(f)+G_P f)-\max_{1\leq j\leq N}(B(f_j)+G_P f_j)\leq \eta+\|G_P\|_{\mF_\eps}.\label{eq: MB appr2}
\end{align}

\medskip

{\bf Step 1}. By Step 2 of the proof of Theorem \ref{thm: inf dim GAR}, we have
\[
\bP(\|G_P\|_{\mF_\eps} > C \sqrt{\sigma^2 K_n/n} ) \leq 2n^{-1}.
\]

{\bf Step 2}. Here we wish to show that on the event $E$, 
\begin{equation}\label{eq: Step 2 - 2}
\bP \left \{  \|\bG_n^e\|_{\mF_{\varepsilon}} > C \{ (b\sigma K_n^{3/2})^{1/2}/(\gamma^{1/q}n^{1/4})+b K_n/(\gamma^{1/q}n^{1/2-1/q}) \}  \mid X_{1}^{n} \right \} \leq 2n^{-1}.
\end{equation}

Fix any $X_{1},\dots,X_{n}$ satisfying (\ref{eq: event1})--(\ref{eq: event3}). Let us write $(\mF-\mF)^2:=\{(f-g)^2:f,g\in\mF\}$. 
Then observe that 
\begin{align*}
\sigma_n^2&:=\sup_{f\in\mF_\eps} P_n f^2 \leq \sup_{f\in\mF_{\varepsilon}} \bE[f(X)^2]+n^{-1/2} \|\bG_n\|_{(\mF-\mF)^2}  \\
&\lesssim (\eps b)^2 +n^{-1/2} \|\bG_n\|_{\mF\cdot\mF} \lesssim \sigma^2/n + b\sigma K_n^{1/2}/(\gamma^{2/q} n^{1/2})+b^2 K_n/(\gamma^{2/q} n^{1-2/q})   \\
&\lesssim  b\sigma K_n^{1/2}/(\gamma^{2/q} n^{1/2})+b^2 K_n/(\gamma^{2/q} n^{1-2/q}), 
\end{align*}
where in the second line, we used the inequality $\|\bG_n\|_{(\mF-\mF)^2} =\sup_{f,g\in \mF}|\bG_n(f-g)^2| \leq 4\|\bG_n\|_{\mF\cdot\mF}$. 
Now, note that conditional on $X_{1}^{n}$,  $\bG_n^e$ is a centered Gaussian process, and $\bE[(\bG_n^e f)^2 \mid X_{1}^{n}] \leq P_n f^2\leq \sigma_n^2$ for all $f \in \mF_\eps$. 
Hence by the Borell-Sudakov-Tsirel'son inequality \citep[see][Proposition A.2.1]{VW96}, 
\[
\bP\{ \|\bG_n^e\|_{\mF_\eps} > \bE[\|\bG_n^e\|_{\mF_\eps} \mid X_{1}^{n}] + \sigma_n \sqrt{2\log n} \mid X_{1}^{n} \} \leq 2 n^{-1}.
\]
To bound $\bE[\|\bG_n^e\|_{\mF_\eps} \mid X_{1}^{n}]$, observe that
\begin{align*}
\|\bG_n^e\|_{\mF_\eps}
&\leq \sup_{f\in\mF_{\eps}}\Big|\frac{1}{\sqrt{n}}\sum_{i=1}^n e_i f(X_i)\Big|+\sup_{f\in\mF_{\eps}}\Big|\frac{1}{\sqrt{n}}\sum_{i=1}^n e_i\cdot P_n f\Big| =:I+II. 
\end{align*}
By Dudley's maximal inequality \citep[see][Corollary 2.2.8]{VW96}, together with the fact that $N(\mF_{\eps},e_{P_{n}}, 2\tau \| F \|_{P_{n},2}) \leq N^{2}(\mF,e_{P_{n}},\tau \| F \|_{P_{n},2}) \leq (A/\tau)^{2v}$, we have 
\begin{align*}
\bE [ I \mid X_{1}^{n} ] &\lesssim \int_{0}^{\sigma_{n} \vee (\sigma/n^{1/2})} \sqrt{1+\log N(\mF_{\eps},e_{P_{n}}, \tau)} d\tau \\
&\lesssim (\sigma_n\vee(\sigma/n^{1/2}))\sqrt{v\log(2n^{1/2} A\|F\|_{P_n,2}/\sigma)} \lesssim (\sigma_n\vee (\sigma/n^{1/2})) K_n^{1/2}.
\end{align*}
Meanwhile, since $\|P_n\|_{\mF_\varepsilon}\leq \sigma_n$ by Jensen's inequality, we have
\[
\bE[II \mid X_{1}^{n}]  \leq \|P_n\|_{\mF_\eps}\cdot\bE\left[\Big|\frac{1}{\sqrt{n}}\sum_{i=1}^n e_i\Big|\right]\lesssim \sigma_n.
\]
Combining these inequalities leads to (\ref{eq: Step 2 - 2}).

\medskip

{\bf Step 3}. Let $Z^{e,\eps}=\max_{1\leq j\leq N}(B(f_j)+\bG_n^{e} f_j)$ and $\tilde{Z}^{\eps} = \max_{1\leq j\leq N}(B(f_j)+G_P f_j)$. We wish to show that on the event $E$, the inequality 
\[
\bP (Z^{e,\eps} \in A \mid X_{1}^{n}) \leq \bP(\tilde{Z}^{\eps} \in A^{\delta}) + \frac{C}{\delta}\left\{ \frac{(b\sigma K_n^{3/2})^{1/2}}{\gamma^{1/q}n^{1/4}} + \frac{b K_n}{\gamma^{1/q}n^{1/2-1/q}}\right \}
\]
holds for every $\delta > 0$ and every Borel subset $A$ of $\R$. 
Let
\[
\Delta := \max_{1\leq j,k\leq N}|\{ P_n (f_j f_k) - (P_n f_j) ( P_n f_k) \}- \{ P (f_j f_k) - (P f_j)(P f_k) \} |,
\]
and observe that
\begin{align*}
&|P_n (f_j f_k) - P (f_j f_k)| \leq n^{-1/2} \|\bG_n\|_{\mF\cdot\mF}, \\
&|(P_n f_j) (P_n f_k)-(P f_j)(P f_k)| \lesssim n^{-1} \|\bG_n\|_{\mF}\cdot\|\bG_n\|_{\mF}+\sigma n^{-1/2} \|\bG_n\|_{\mF}.
\end{align*}
Hence as $K_n \leq n$, it is not difficult to check that on the event $E$, 
\[
\Delta\lesssim b\sigma K_n^{1/2}/(\gamma^{2/q}n^{1/2})+b^2 K_n/(\gamma^{2/q}n^{1-2/q}). 
\]
The assertion of this step now follows from Theorem \ref{thm: mb finite dimensional} (recall $\log N \lesssim K_{n}$).

\medskip

{\bf Step 4}. Take
\[
\delta=\delta_n^{(2)} = \frac{(b\sigma K_n^{3/2})^{1/2}}{\gamma^{1 + 1/q} n^{1/4}} + \frac{b K_n}{\gamma^{1 + 1/q} n^{1/2-1/q}}.
\]
Then the desired inequality (\ref{eq: desired inequality2}) (with suitable $C_{3},C_{4}$) follows from combining (\ref{eq: MB appr1}), (\ref{eq: MB appr2}), Steps 1,2, and 3 with this choice of $\delta$. 
\end{proof}

\subsection{Proof of Theorem \ref{thm: inf dim EB}}

Here $C$ denotes a positive constant depending only on $q$; $C$ may change from place to place. In addition, to ease the notation, we write $a\lesssim b$ if $a \leq Cb$. In the proof below, we find an event $E\in\sigma(X_1^n)$ such that $\bP(E)\geq 1-\gamma-n^{-1}$, and on this event, the inequality
\begin{equation}\label{eq: desired inequality3}
\bP(Z^*\in A \mid X_1^n)\leq \bP\{Z^e\in A^{C(\eta+\delta_n^{(3)})} \mid X_1^n\} + C(\gamma + n^{-1})
\end{equation}
holds for every Borel subset $A$ of $\R$ where $Z^e=\sup_{f\in\mF}(B(f)+\bG_n^e f)$. Combining this inequality with (\ref{eq: desired inequality2}), which is established in the proof of Theorem \ref{thm: inf dim MB} (and which holds on a possibly different event $E'\in\sigma(X_1^n)$ satisfying $\bP(E')\geq 1-\gamma - n^{-1}$), the proof is completed by applying Lemma \ref{lem: conditional strassen}.

We first specify the event $E$. We use the same notation as introduced in Step 1 of the proof of Theorem \ref{thm: inf dim GAR}. Then
\begin{align}
&0\leq \sup_{f\in\mF}(B(f)+\bG_n^* f) - \max_{1\leq j\leq N}(B(f_j) - \bG_n^* f_j) \leq \eta + \|\bG_n^*\|_{\mF_\varepsilon},\label{eq: EB approximation1}\\
&0\leq \sup_{f\in\mF}(B(f)+\bG_n^e f) - \max_{1\leq j\leq N}(B(f_j) - \bG_n^e f_j) \leq \eta + \|\bG_n^e\|_{\mF_\varepsilon}.\label{eq: EB approximation2}
\end{align}
In addition, as in the proof of Theorem \ref{thm: inf dim MB}, with probability at least $1-\gamma/4$,
\begin{equation}
\label{eq: EB event1}
\|\bG_n\|_\mF \lesssim \sigma K_n^{1/2}/\gamma^{1/q}+b K_n/(\gamma^{1/q}n^{1/2-1/q}); 
\end{equation}
with probability at least $1-\gamma/4$,
\begin{equation}
\label{eq: EB event2}
\|\bG_n\|_{\mF\cdot\mF} \lesssim b\sigma K_n^{1/2}/(\gamma^{2/q})+b^2K_n/(\gamma^{2/q}n^{1/2-2/q}),
\end{equation}
and with probability at least $1-n^{-1}$,
\begin{equation}\label{eq: EB event3}
\|F\|_{P_n,2}\leq n^{1/2}\|F\|_{P,2}.
\end{equation}
Here $\mF\cdot\mF=\{f\cdot g:f,g\in\mF\}$. Moreover, by the triangle inequality,
\[
\max_{1\leq j\leq N}\sum_{i=1}^n|f_j(X_i) - P_n f_j|^3 \lesssim \max_{1\leq j\leq N}\sum_{i=1}^n|f_j(X_i)|^3
\]
and applying Lemma \ref{lem: deviation ineq nonnegative}, we have with probability at least $1-\gamma/4$,
\[
\max_{1\leq j\leq N}\sum_{i=1}^n|f_j(X_i)|^3 \lesssim \bE\left[\max_{1\leq j\leq N}\sum_{i=1}^n |f_j(X_i)|^3\right] + \gamma^{-3/q}\|M^3\|_{q/3},
\]
where $M := \max_{1\leq i\leq n}\max_{1\leq j\leq N}|f_j(X_i)|\leq  \max_{1\leq i\leq n}F(X_i)$ is such that $\|M^3\|_{q/3}\lesssim n^{3/q}b^3$. In addition, by Lemma \ref{lem: maximal ineq nonnegative},
\[
\bE\left[\max_{1\leq j\leq N}\sum_{i=1}^n |f_j(X_i)|^3\right] \lesssim n\sigma^2 b+\bE[M^3]\log N\lesssim n\sigma^2 b+n^{3/q}b^3K_n.
\]
Therefore, with probability at least $1-\gamma/4$,
\begin{equation}\label{eq: EB event4}
\max_{1\leq j\leq N}\sum_{i=1}^n |f_j(X_i) - P_n f_j|^3/n \lesssim \sigma^2 b+b^3 K_n/(\gamma^{3/q} n^{1-3/q}).
\end{equation}
Finally, by Markov's inequality, with probability at least $1-\gamma/4$,
\begin{equation}\label{eq: EB event5}
\max_{1\leq i\leq n}\max_{1\leq j\leq N}|f_j(X_i)-P_n f_j|\lesssim \max_{1\leq i\leq n}F(X_i) \lesssim \gamma^{-1/q}n^{1/q}b.
\end{equation}
Define $E$ as the intersection of the events in (\ref{eq: EB event1})-(\ref{eq: EB event5}). Then $E\in\sigma(X_1^n)$ and $\bP(E)\geq 1- \gamma - n^{-1}$. In the rest of the proof, which is divided into several steps, we prove (\ref{eq: desired inequality3}) for each fixed $X_1,\dots,X_n$ satisfying (\ref{eq: EB event1})--(\ref{eq: EB event5}).

\medskip

{\bf Step 1}. By Step 2 in the proof of Theorem \ref{thm: inf dim MB}, on the event $E$,
\[
\bP\left\{ \|\bG_n^e\|_{\mF_{\varepsilon}}> C \{ (b\sigma K_n^{3/2})^{1/2}/(\gamma^{1/q}n^{1/4})+b K_n/(\gamma^{1/q}n^{1/2-1/q}) \} \mid  X_{1}^{n}  \right \} \leq 2n^{-1}.
\]

{\bf Step 2}. Here we wish to show that on the event $E$,
\begin{equation}
\label{eq: Step 2 - 3}
\bP\left\{\|\bG_n^*\|_{\mF_{\varepsilon}}> C\{ (b\sigma K_n^{3/2})^{1/2}/(\gamma^{1/q}n^{1/4})+ b K_n/(\gamma^{1/q} n^{1/2-1/q}) \}  \mid X_{1}^{n} \right\} \leq n^{-1}.
\end{equation}

Note that conditional on $X_1^n$, $\bG_n^*$ is the empirical process associated with $n$ i.i.d. observations from the empirical distribution $P_n$. When restricted to the domain $\{X_1,\dots,X_n\}$, the class of functions $\mF$ has a constant envelope $\max_{1\leq i\leq n}F(X_i) \lesssim \gamma^{-1/q}n^{1/q}b$. 
Moreover, by the same arguments as those used in Step 2 of the proof of Theorem \ref{thm: inf dim MB},
\[
\sigma_n^2 :=\sup_{f\in\mF_\varepsilon}P_n f^2 \lesssim b\sigma K_n^{1/2}/(\gamma^{2/q}n^{1/2}) + b^2 K_n/(\gamma^{2/q} n^{1-2/q}).
\]
Hence the inequality (\ref{eq: Step 2 - 3}) follows from application of Talagrand's inequality (Lemma \ref{lem: talagrand inequality})  with $t=\log n$.

\medskip

{\bf Step 3}. Let $Z^{*,\varepsilon} = \max_{1\leq j\leq N}(B(f_j) + \bG_n^* f_j)$ and $Z^{e,\varepsilon} = \max_{1\leq j\leq N} (B(f_j) + \bG_n^e f_j)$. Here we apply Theorem \ref{thm: fin dim main} to show that whenever
\begin{equation}\label{eq: delta requirement 3}
\delta\geq C\left(\frac{b\log N}{\gamma^{1/q} n^{1/2-1/q}} + \frac{\sigma(\log^{3/2} N)\cdot (\log n)}{n^{1/2}}\right)
\end{equation}
for some sufficiently large $C>0$, on the event $E$, the inequality
\begin{equation*}
\bP(Z^{*,\varepsilon}\in A  \mid X_{1}^{n} )\leq \bP(Z^{e,\varepsilon}\in A^{C\delta}  \mid X_{1}^{n}) + C\left( \frac{b \sigma^2 K_n^2}{\delta^3n^{1/2}} + \frac{b^3 K_n^3}{\delta^3 \gamma^{3/q}n^{3/2-3/q}} + \frac{1}{n} \right)
\end{equation*}
holds for every $\delta > 0$ and every Borel subset $A$ of $\R$. 
Let $\widetilde{X}_i = (f_j(X_i) - P_n f_j)_{1\leq j\leq N}$, $1\leq i\leq n$, and let $\widetilde{Y} = (\bG_n^e f_{j})_{1\leq j\leq N}$.
Then
\begin{align*}
L_n&=\max_{1\leq j\leq N}\sum_{i=1}^n |\tilde{X}_{i j}|^3/n \lesssim \sigma^2 b+b^3 K_n/(\gamma^{3/q} n^{1-3/q}), \\
M_{n,X}(\delta)&=n^{-1} \sum_{i=1}^n\max_{1\leq j\leq N}|\tilde{X}_{i j}|^3\cdot 1\left\{\max_{1\leq j\leq N}|\tilde{X}_{i j}|>\delta\sqrt{n}/\log N\right\}=0.
\end{align*}
The last equality follows from (\ref{eq: EB event5}) since $\delta\sqrt{n}/\log N\geq C\gamma^{-1/q}n^{1/q}b$. Moreover, $\bE[\tilde Y_j^2]\leq P_n f_j^2 \leq \sigma^2 + n^{-1/2}\|\bG_n\|_{\mathcal F\cdot \mathcal F}$ for all $1\leq j\leq N$, and so by the same argument as that used in Step 4 of the proof of Theorem \ref{thm: inf dim GAR}, we have
\begin{align*}
M_{n,Y}(\delta)&=\bE \left[ \max_{1\leq j\leq N}|\tilde{Y}_{j}|^3\cdot 1\left\{\max_{1\leq j\leq N}|\tilde{Y}_{j}|>\delta\sqrt{n}/\log N\right\} \mid X_1^n \right]\\
& \lesssim n^{-2}(\delta\sqrt{n}/\log N)^3,
\end{align*}
since $\delta\geq C(\sigma^2 + n^{-1/2}\|\bG_n\|_{\mathcal F\cdot \mathcal F})^{1/2} n^{-1/2}(\log^{3/2} N)\cdot(\log n)$ for sufficiently large $C$.
The assertion of this step then follows from Theorem \ref{thm: fin dim main}.

\medskip

{\bf Step 4}. In the previous step, take
\[
\delta = C'\left\{\frac{(b\sigma^2 K_n^2)^{1/3}}{\gamma^{1/3} n^{1/6}} + \frac{b K_n}{\gamma^{1/3 + 1/q}n^{1/2-1/q}}\right\}
\]
where $C' > 0$ is a large constant that can be chosen to depend only on $q$. It is not difficult to check that for this choice of $\delta$, (\ref{eq: delta requirement 3}) holds under the condition $K_n^3\leq n$. The desired inequality (\ref{eq: desired inequality3}) then follows from combining (\ref{eq: EB approximation1}), (\ref{eq: EB approximation2}), Steps 1, 2, and 3 with this choice of $\delta$. 
\qed

\subsection{Proof of Lemma \ref{lem: AC for nonzero mean}}

We begin with proving the following lemma. 
\begin{lemma}
\label{lem: AC nonzero mean fin dim}
Let $X=(X_{1},\dots,X_{p})^{T}$ be a possibly non-centered Gaussian random vector with $\sigma_{j}^{2} := \Var (X_{j}) > 0, 1 \leq  j \leq p$. 
Then for every $\eps > 0$,
\[
\sup_{t \in \R} \bP (| \max_{1 \leq j \leq p} X_{j} -t| \leq \eps) \leq \frac{2\eps}{\underline{\sigma}} (\sqrt{2 \log p} + 2),
\]
where $\underline{\sigma} = \min_{1 \leq j \leq p} \sigma_{j}$.
\end{lemma}

The lemma follows from the following result due essentially to Nazarov \cite{Nazarov03}; see also \cite{Klivans08}. 

\begin{lemma}[Nazarov's inequality]
\label{lem: Nazarov}
Let $W$ be a standard Gaussian random vector in $\R^{m}$, that is, $W \sim N(0,I)$. Let $A \subset \R^{m}$ be the intersection of $p$ half-spaces (a half-space in $\R^{m}$ is the set of form $\{ w \in \R^{m} : \alpha^{T} w \leq t \}$ for some $\alpha \in \R^{m}$ with $\| \alpha \| = 1$ and $t \in \R$). Then 
\[
\lim_{\delta \downarrow 0} \frac{1}{\delta} \bP (W \in A^{\delta} \backslash A) \leq \sqrt{2\log p}+2,
\]
where the limit on the left-hand side exists. 
\end{lemma}

\begin{proof}[Proof of Lemma \ref{lem: AC nonzero mean fin dim}]
It is clear that the distribution of $\max_{1 \leq j \leq p} X_{j}$ is absolutely continuous, so let $f(\cdot)$ denote its density. 
Let $W$ be a standard Gaussian random vector in $\R^{p}$, and let $\mu = (\mu_1,\dots,\mu_p)^{T}= \Ep [X]$ and $\Sigma = \Ep[(X-\mu)(X-\mu)^{T}]$.
Then $X  \stackrel{d}{=} \Sigma^{1/2}W+\mu$, so that denoting by $\sigma_{j}a_{j}^{T}$ (where $a_{j} \in \R^{p}$ with $\| a_j \| = 1$) the $j$-th row of $\Sigma^{1/2}$, we obtain 
\[
\max_{1 \leq j \leq p} (\Sigma^{1/2}W+\mu)_{j} \leq t \Leftrightarrow a_{j}^{T}W \leq (t-\mu_{j})/\sigma_{j}, 1 \leq \forall j \leq p.
\]
Let $A_{t} = \{ w \in \R^{p} : a_{j}^{T} w \leq (t-\mu_{j})/\sigma_{j}, 1 \leq \forall j \leq p \}$ for $t \in \R$, and observe that
\[
f(t) = \lim_{\eps \downarrow 0} \frac{1}{\eps} \bP (W \in A_{t+\eps} \backslash A_{t}) \quad  a.e. \ t \in \R.
\] 
Moreover, since
\[
A_{t+\eps} \subset  \{ w \in \R^{p} : a_{j}^{T} w \leq (t-\mu_{j})/\sigma_{j} + \eps/\underline{\sigma}, 1 \leq \forall j \leq p \},
\]
we have by Lemma \ref{lem: Nazarov},
\[
\frac{1}{\eps} \bP (W \in A_{t+\eps} \backslash A_{t}) \leq \frac{1}{\eps}\bP(W \in A^{\eps/\underline{\sigma}}_{t} \backslash A_{t}) + o(1) \leq \frac{1}{\underline{\sigma}} (\sqrt{2 \log p} + 2) + o(1), \ \eps \downarrow 0,
\]
which leads to $f(t) \leq (1/\underline{\sigma}) (\sqrt{2\log p} + 2)$ a.e.
\end{proof}
We are now in position to prove Lemma \ref{lem: AC for nonzero mean}.

\begin{proof}[Proof of Lemma \ref{lem: AC for nonzero mean}]
Pick any $\delta > 0$, and let $\{ t_{1},\dots,t_{N} \}$ be a $\delta$-net of $(T,d)$ with $N=N(T,d,\delta)$. 
Then 
\[
| \sup_{t \in T} X(t) - \max_{1 \leq j \leq N} X(t_{j}) | \leq \sup_{(s,t) \in T_{\delta}} | X(t) - X(s) | =: \zeta,
\]
so that for every $x \in \R, \eps, r' > 0$, 
\[
\bP (| \sup_{t \in T} X(t) -x| \leq \eps) \leq \bP (| \sup_{1 \leq j \leq N} X(t_{j}) -x| \leq \eps + r') + \bP (  \zeta > r' ).
\]
By Lemma \ref{lem: AC nonzero mean fin dim}, the first term on the right-hand side is bounded by 
\[
2(1/\underline{\sigma})(\eps + r') (\sqrt{2 \log N} + 2).
\]
On the other hand, by the Borell-Sudakov-Tsirel'son inequality, for every $r > 0$, 
\[
\bP ( \zeta > \bE [\zeta] + r\delta) \leq e^{-r^{2}/2}.
\]
By taking $r' = \bE [\zeta] + r\delta= \phi(\delta) + r\delta$, we obtain the desired conclusion. 
\end{proof}

\section{Proofs for Section \ref{sec: auxiliary couplings}}
\label{sec: proofs auxiliary couplings}

We begin with proving the following lemma.

\begin{lemma}
\label{thm: smooth approximation}
Let $\delta>0$. For every Borel subset $A$ of $\R$, there exists a smooth function $g:\R \to \R$ such that $\| g' \|_{\infty}\leq\delta^{-1}, \| g'' \|_{\infty}\leq K\delta^{-2}, \| g'''\|_{\infty} \leq K\delta^{-3}$, where $K$ is an absolute constant, and $1_A(t)\leq g(t)\leq 1_{A^{3\delta}}(t)$ for all $t\in \R$.
\end{lemma}
\begin{proof}
The proof is essentially similar to that of Lemma 18 in Chapter 10 of \cite{P02} with the exception that we employ a compactly supported smoother. Let $\rho$ denote the Euclidean distance on $\R$, and consider the function $h(t)=(1-\rho(t,A^\delta)/\delta)_{+}$. Observe that $h$ is a bounded Lipschitz function with Lipschitz constant $\delta$. Let $\varphi:\R\to\R$ be the function defined by $\varphi(t)=C\exp(1/(t^2-1))$ for $|t|\leq 1$ and $\varphi(t)=0$ for $|t|>1$, where the constant $C$ is chosen in such a way that $\int_{\R}\varphi(t)dt=1$. Note that $\varphi$ is infinitely differentiable with support $[-1,1]$. Define $g:\R\to\R$ by
\[
g(t)=\int_{\R}h(t+\delta z)\varphi(z)dz = \delta^{-1}\int_{\R}h(y)\varphi(\delta^{-1}(y-t))dy. 
\]
Then it is routine to verify that $g$ is infinitely differentiable and $\| g' \|_{\infty} \leq \delta^{-1}, \| g'' \|_{\infty} \leq K \delta^{-2}, \| g''' \|_{\infty} \leq K\delta^{-3}$.
In addition, for $t\in A$, $h(t+\delta z)=1$ if $|z|\leq 1$, and $\varphi(z)=0$ if $|z|>1$. Hence $1_A(t) \leq g(t)$. Meanwhile, for  $t \notin A^{3\delta}$, $h(t+\delta z)=0$ if  $|z| \leq 1$, and $\varphi(z)=0$ if $|z|>1$. Hence $g(t)\leq 1_{A^{3\delta}}(t)$. 
\end{proof}

\begin{proof}[Proof of Theorem \ref{thm: fin dim main}]
Here we write $a \lesssim b$ if there exists a universal constant $C > 0$ such that $a \leq Cb$. 
Fix $\delta>0$, and let $\beta=\delta^{-1}\log p$. Since $p\geq 2$, we have $1/\delta\lesssim \beta$. Let $A$ be a Borel subset of $\R$. Letting $e_\beta=\beta^{-1}\log p(=\delta)$ and using Lemma \ref{thm: smooth approximation}, we can construct a smooth function $g:\R\to \R$ such that $\|g'\|_\infty\leq \delta^{-1}$, $\|g''\|_\infty\leq K\delta^{-2}$, $\|g'''\|_\infty\leq K\delta^{-3}$ for some absolute constant $K > 0$, and $1_{A^{e_\beta}}(t)\leq g(t)\leq 1_{A^{e_\beta+3\delta}}(t)$ for all $t \in \R$. 
In addition, let $\bar{\mu}=\sum_{i=1}^n \mu_i$ and consider the function $F_\beta:\R^p\to \R$ defined by $F_\beta(x)=\beta^{-1}\log(\sum_{j=1}^p e^{\beta(x_j+\bar{\mu}_j)}), \ x\in\R^p$. Then it is seen that 
$\max_{1\leq j\leq p}x_j\leq F_\beta(x-\bar{\mu})\leq \max_{1\leq j\leq p}x_j+e_\beta$ for all $x \in \R^{p}$. 
Hence
\begin{equation}
\label{eq: starting inequality}
\bP(Z\in A)\leq \bP\left(F_\beta\left(n^{-1/2}\textstyle{\sum_{i=1}^n} \tilde{X}_i\right)\in A^{e_\beta}\right).
\end{equation}

Next, let $m = g \circ F_{\beta}$. Then, as in the proof of Lemma 5.1 in \cite{CCK3} (see also \cite{C05a, C05b}), there exist functions $U_{j k l}:\R^p\to \R, \ 1\leq j,k,l\leq p$ such that
\begin{align*}
&|\partial_j\partial_k\partial_l m(x)|\leq U_{j k l}(x), \ \forall x \in \R^{p}, \\
&\textstyle{\sum_{j,k,l=1}^p} U_{j k l}(x)\lesssim (\delta^{-3}+\beta\delta^{-2}+\beta^2\delta^{-1})\lesssim \beta^2\delta^{-1}, \ \forall x \in \R^{p},  \\
&U_{j k l}(x)\lesssim U_{j k l}(x+y)\lesssim U_{j k l}(x), \ \forall x,y \in \R^{p} \ \text{with} \ \max_{1 \leq j \leq p} |y_{j}| \leq \beta^{-1}. 
\end{align*}
Hence proceeding as in Step 1 of the proof of Lemma 5.1 in \cite{CCK3} and observing that the term $\int_0^1 \omega(t)\bE[h(Z^{(n)},6)]dt$ in that paper is trivially bounded  by a universal constant, one can show that for some universal constant $c> 0$,
\begin{align*}
&\left | \bE \left [ m\left (n^{-1/2} {\textstyle \sum}_{i=1}^n \tilde{X}_i \right )\right]-\bE \left [ m\left (n^{-1/2} {\textstyle \sum}_{i=1}^n \tilde{Y}_i \right )\right] \right |\\
&\quad \lesssim \frac{\log^2 p}{\delta^3 \sqrt{n}}\cdot \{ L_n+M_{n,X}(c\delta)+M_{n,Y}(c\delta) \} =: I, 
\end{align*}
which implies that for some universal constant $C$, 
\begin{align*}
&\bP\left(F_\beta\left(n^{-1/2}\textstyle{\sum_{i=1}^n} \tilde{X}_i\right)\in A^{e_\beta}\right) \leq \bE\left[m\left(n^{-1/2}\textstyle{\sum_{i=1}^n}\tilde{X}_i\right)\right] \\
&\leq \bE\left[m\left(n^{-1/2}\textstyle{\sum_{i=1}^n}\tilde{Y}_i\right)\right] + C I \leq \bP\left\{ F_\beta\left(n^{-1/2}\textstyle{\sum_{i=1}^n} \tilde{Y}_i\right)\in A^{e_\beta+3\delta}\right \} + C I \\
&\leq \bP\left(\tilde{Z}\in A^{2e_\beta+3\delta}\right)+C I.
\end{align*}
Combining this inequality with (\ref{eq: starting inequality}) leads to the conclusion of the theorem. 
\end{proof}

\begin{proof}[Proof of Theorem \ref{thm: mb finite dimensional}]
Since $p\geq 2$, the assertion is trivial if $\Delta/\delta^2>1$. Therefore, throughout the proof, we will assume that $\Delta/\delta^2\leq 1$. Let $\beta>0$, and define $F_\beta:\R^p\to \R$ by $F_\beta(x)=\beta^{-1}\log(\sum_{j=1}^p e^{\beta(x_j+\mu_j)})$ where $x=(x_1,\dots,x_p)^{T}$ and $\mu=(\mu_1,\dots,\mu_p)^{T}$. As in the proof of Theorem \ref{thm: fin dim main}, it can be shown that for every $g\in C^2(\R)$, the function $m = g \circ F_{\beta}$  satisfies the inequality
\[
\sum_{j,k=1}^p |\partial_j\partial_k m(x)|\leq \|g''\|_\infty+2\|g'\|_\infty \beta
\]
for all $x\in \R^p$.
Hence using the same arguments as those used in the proof of Theorem 1 and Comment 1 in \cite{CCK2} with  $X$ and $Y$ replaced by  $X-\mu$ and $Y-\mu$, respectively, we have
\[
\left|\bE\left[g\left(\max_{1\leq j\leq p}X_j\right)\right]-\bE\left[g\left(\max_{1\leq j\leq p}Y_j\right)\right]\right|\leq \|g''\|_\infty \Delta/2+2\|g'\|_\infty \sqrt{2\Delta\log p}.
\]

Now, take any Borel subset $A$ of $\R$. By Lemma \ref{thm: smooth approximation}, we can construct a function $g\in C^2(\R)$ such that $\|g'\|_\infty\leq \delta^{-1}$ and $\|g''\|_\infty\leq K\delta^{-2}$ for some absolute constant $K$, and $1_A(t)\leq g(t)\leq 1_{A^{3\delta}}(t)$ for all $t\in\R$. 
For this  $g$ and some absolute constant $C$, we have
\begin{align*}
&\bP\Big(\max_{1\leq j\leq p}X_j\in A\Big)\leq \bE\Big[g\Big(\max_{1\leq j\leq p}X_j\Big)\Big] \\
&\quad \leq \bE\Big[g\Big(\max_{1\leq j\leq p}Y_j\Big)\Big]+C(\Delta\delta^{-2}+\delta^{-1}\sqrt{\Delta\log p})\\
&\quad \leq \bP\Big(\max_{1\leq j\leq p}Y_j\in A^{3\delta}\Big)+C(\Delta\delta^{-2}+\delta^{-1}\sqrt{\Delta\log p})\\
&\quad \leq \bP\Big(\max_{1\leq j\leq p}Y_j\in A^{3\delta}\Big)+C\sqrt{(\Delta/\delta^2)\log p}
\end{align*}
where the last line follows from the fact that $\Delta/\delta^2\leq 1$ and $p\geq 2$.
The conclusion of the theorem follows from replacing $\delta$ by $\delta/3$.
\end{proof}


\section{Some technical tools}
\label{sec: technical tools}

\begin{lemma}
\label{concentration}
Let $X_1,\dots,X_n$ be i.i.d. random variables taking values in a measurable space $(S,\mathcal{S})$ with common distribution $P$. Let $\mF$ be a pointwise measurable class of functions $f:S\to \R$, to which a measurable envelope $F$ is attached. Consider the empirical process $\bG_n f=n^{-1/2}\sum_{i=1}^n (f(X_i)-Pf)$, $f\in\mF$. Let $\sigma^2>0$ be a constant such that $\sup_{f\in\mF}P f^2\leq \sigma^2\leq \|F\|_{P,2}^2$. Let $M=\max_{1\leq i\leq n} F(X_i)$. Suppose that $F \in \mL^{q} ( P )$ for some $q \geq 2$. Then for every $t \geq 1$, with probability $> 1-t^{-q/2}$,
\begin{multline*}
\| \bG_{n} \|_{\mF} \leq (1+\alpha) \bE [ \| \bG_{n} \|_{\mF} ] +K_{q} \Big \{ (\sigma + n^{-1/2} \| M \|_{q}) \sqrt{t} \\
+  \alpha^{-1}  n^{-1/2} \| M \|_{2}t \Big \}, \ \forall \alpha > 0,
\end{multline*}
where $K_{q}> 0$ is a constant that depends only $q$.
\end{lemma}
\begin{proof}
The lemma is essentially due to \cite{BBLM05}, Theorem 12. See Theorem 5.1 in \cite{CCK1} for the version stated here.
\end{proof}

\begin{lemma}
\label{cor: maximal}
Consider the setting of Lemma \ref{concentration}. In addition, suppose that there exist constants $A \geq e$ and $v \geq 1$ such that $\sup_{Q} N(\mF,e_{Q},\varepsilon \| F \|_{Q,2}) \leq (A/\varepsilon)^{v}, \ 0 < \varepsilon \leq 1$.
Then
\begin{equation*}
\bE [ \| \bG_{n} \|_{\mF} ] \leq K \left \{ \sqrt{v\sigma^{2} \log \left ( \frac{A \| F \|_{P,2}}{\sigma} \right ) } + \frac{v\| M \|_{2}}{\sqrt{n}} \log \left ( \frac{A \| F \|_{P,2}}{\sigma} \right ) \right \},
\end{equation*}
where $K$ is an absolute constant.
\end{lemma}
\begin{proof}
See Corollary 5.1 in \cite{CCK1}.
\end{proof}

\begin{lemma}[Talagrand's inequality]
\label{lem: talagrand inequality}
Consider the setting of Lemma \ref{cor: maximal}, but suppose now that the envelope $F$ is bounded by a constant $b > 0$, and let $\sigma^2 > 0$ be a constant such that $\sup_{f\in \mF}Pf^2\leq \sigma^2\leq b^2$. If $b^2v\log(Ab/\sigma)\leq n\sigma^2$, then for every $0 < t \leq n\sigma^2/b^2$,
\[
\bP\left\{ \|\bG_n\|_{\mF}>K\sigma \sqrt{t\vee(v\log(Ab/\sigma)}\right \} \leq e^{-t},
\]
where $K$ is an absolute constant.
\end{lemma}
\begin{proof}
This form of Talagrand's inequality is taken from Theorem B.1 in \cite{CCK4};  the original references go back to \cite{T96}, \cite{M00}, and \cite{GG01}.
\end{proof}



\begin{lemma}
\label{lem: maximal ineq nonnegative}
Let $X_1,\dots,X_n$ be independent random vectors in $\R^p$ with $p\geq 2$ such that $X_{ij} \geq 0$ for all $i=1,\dots,n$ and $j=1,\dots,p$. Define $Z := \max_{1 \leq j \leq p} \sum_{i=1}^{n} X_{ij}$ and $M := \max_{1 \leq i \leq n} \max_{1 \leq j \leq p}  X_{ij} $. Then
\[
\Ep [Z] \leq K \left( \max_{1\leq j\leq p}\Ep [ {\textstyle \sum}_{i=1}^n X_{ij} ]+\Ep [M] \log p \right),
\]
where $K$ is an absolute constant. 
\end{lemma}
\begin{proof}
See Lemma 9 in \cite{CCK2}.
\end{proof}

\begin{lemma}
\label{lem: deviation ineq nonnegative}
Assume the setting of Lemma \ref{lem: maximal ineq nonnegative}. Then
for every $\eta > 0, s \geq 1$ and $t > 0$, 
\[
\Pr ( Z  \geq (1+\eta) \Ep[Z] + t ) \leq K  \Ep[M^{s}]/t^{s},
\]
where $K = K(\eta,s)$ is a constant that depends only on $\eta,s$. 
\end{lemma}
\begin{proof}
See Lemma A.5 in \cite{CCK3}.
\end{proof}

\begin{lemma}
\label{lem: truncated moment}
Let $\xi$ be a nonnegative random variable such that $\Pr(\xi >x)\leq A e^{-x/B}$ for all $x>0$ and for some constants $A, B>0$. Then for every $t>0$, $\Ep[\xi ^31\{\xi >t\}]\leq 6A(t+B)^3e^{-t/B}$. 
\end{lemma}
\begin{proof}
See Lemma A.8 in \cite{CCK3}.
\end{proof}


\section*{Acknowledgments} 
We would like thank an anonymous referee for valuable comments that helped improve upon the quality of the paper. K. Kato is supported by the Grant-in-Aid for Scientific Research (C) (15K03392) from the Japan Society for the Promotion of Science.  V. Chernozhukov is supported  by a grant from the National Science Foundation.

\end{document}